\documentclass[12pt,reqno]{amsart}

\usepackage{amsbsy}
\usepackage{amscd}
\usepackage{amsfonts}
\usepackage{amsmath}
\usepackage{amssymb}
\usepackage{amsthm}
\usepackage{amsxtra}

\usepackage[utf8]{inputenc}
\usepackage[margin=1in]{geometry}

\usepackage{datetime}
\usepackage{tikz-cd}
\usepackage{tikz}
\usepackage{todonotes}
\usepackage{verbatim}

\theoremstyle{definition}

\newtheorem{thm}{Theorem}[section]
\newtheorem*{thm*}{Theorem}

\newtheorem*{conj*}{Conjecture}

\newtheorem*{conv*}{Convention}

\newtheorem*{cor*}{Corollary}

\newtheorem*{defn*}{Definition}

\newtheorem*{exa*}{Example}
\newtheorem{exc}[thm]{Exercise}
\newtheorem*{exc*}{Exercise}

\newtheorem*{fact*}{Fact}
\newtheorem{lem}[thm]{Lemma}
\newtheorem*{lem*}{Lemma}

\newtheorem*{prob*}{Problem}

\newtheorem*{prop*}{Proposition}

\newtheorem*{ques*}{Question}
\newtheorem{rmk}[thm]{Remark}
\newtheorem*{rmk*}{Remark}

\newcommand{\mc}{\mathcal}

\newcommand{\C}{\mathbb{C}}

\renewcommand{\H}{\mathbb{H}}

\renewcommand{\P}{\mathbb{P}}
\newcommand{\Q}{\mathbb{Q}}

\newcommand{\R}{\mathbb{R}}

\newcommand{\Z}{\mathbb{Z}}

\newcommand{\cC}{\mathcal{C}}

\newcommand{\cM}{\mathcal{M}}

\newcommand{\al}{\alpha}

\newcommand{\Gam}{\Gamma}
\newcommand{\gam}{\gamma}

\newcommand{\eps}{\varepsilon}

\newcommand{\sig}{\sigma}

\newcommand{\Om}{\Omega}
\newcommand{\om}{\omega}

\newcommand{\ra}{\rightarrow}

\newcommand{\ol}{\overline}
\newcommand{\pr}{\prime}

\newcommand{\wt}{\widetilde}

\newcommand{\sm}{\setminus}

\DeclareMathOperator{\GL}{GL}
\DeclareMathOperator{\SL}{SL}

\DeclareMathOperator{\Aut}{Aut}

\DeclareMathOperator{\lcm}{lcm}

\DeclareMathOperator{\ord}{ord}

\DeclareMathOperator{\re}{Re}

\DeclareMathOperator{\sgn}{sgn}

\DeclareMathOperator{\Tr}{Tr}

\newcommand{\be}{\begin{equation*}}
\newcommand{\ee}{\end{equation*}}
\newcommand{\bex}{\begin{exc}}
\newcommand{\eex}{\end{exc}}
\newcommand{\bpf}{\begin{proof}}
\newcommand{\epf}{\end{proof}}

\settimeformat{hhmmsstime}

\title{Uniqueness of the Veech 14-gon}
\author[Karl Winsor]{Karl Winsor \\ \\ \monthname[\the\month] \the\day, \the\year}

\begin{document}

\maketitle

\begin{abstract}
We obtain the first complete classification result for algebraically primitive Teichm\"{u}ller curves in a stratum component in genus greater than $2$. Specifically, we show that the Veech $14$-gon generates the unique algebraically primitive Teichm\"{u}ller curve in the hyperelliptic component of the stratum $\Omega\mathcal{M}_3(2,2)$.
\end{abstract}


\section{Introduction} \label{sec:intro}

The moduli space $\cM_g$ of closed Riemann surfaces of genus $g$ is swept out by complex geodesics $f : \H \ra \cM_g$ which are isometrically immersed for the Teichm\"{u}ller metric on $\cM_g$. The image of a typical complex geodesic is dense. However, occasionally it may happen that $f$ is stabilized by a lattice $\Gam$ in $\Aut(\H)$, in which case $f$ covers an algebraic {\em Teichm\"{u}ller curve}. The study of Teichm\"{u}ller curves has rich connections with the dynamics of polygonal billiards and interval exchanges, the theory of real multiplication and torsion packets in families of Jacobians, and the flat geometry of holomorphic $1$-forms.

Let $\Om\cM_g \ra \cM_g$ be the bundle of nonzero holomorphic $1$-forms, denoted by pairs $(X,\om)$ with $X \in \cM_g$ and $0 \neq \om \in \Om(X)$. The space $\Om\cM_g$ has a natural action of $\GL^+(2,\R)$ arising from the Teichm\"{u}ller geodesic flow and scaling by complex numbers. Up to passing to a double cover, every Teichm\"{u}ller curve arises from the projection of a closed $\GL^+(2,\R)$-orbit in $\Om\cM_g$, so we will only consider Teichm\"{u}ller curves arising in this way. The $\GL^+(2,\R)$-action preserves each stratum $\Om\cM_g(k_1,\dots,k_n)$ of $1$-forms with $n$ distinct zeros of multiplicities $k_1,\dots,k_n$, so one can also refer to Teichm\"{u}ller curves in a stratum.

A Teichm\"{u}ller curve in $\Om\cM_g$ is {\em primitive} if it does not arise from a Teichm\"{u}ller curve in $\Om\cM_h$ for some $h < g$ by a covering construction. The first infinite family of primitive Teichm\"{u}ller curves was discovered in \cite{Vee:Teich}, and their underlying Riemann surfaces arise from billiards in regular $n$-gons via an unfolding construction \cite{KZ:unfolding}. An infinite family of primitive Teichm\"{u}ller curves in genus $2$ was discovered independently in \cite{Cal:Teich} and \cite{McM:HMS}. A complete classification of primitive Teichm\"{u}ller curves in genus $2$ appears in \cite{McM:discriminant}, \cite{McM:torsion}.

The traces of elements of the stabilizer of $f$ generate a number field of degree at most $g$, called the {\em trace field} of $f$. A Teichm\"{u}ller curve in $\Om\cM_g$ is {\em algebraically primitive} if its trace field has the maximum possible degree $g$. The relative abundance of Teichm\"{u}ller curves depends strongly on the degree of the trace field. Teichm\"{u}ller curves with trace field $\Q$ are dense in $\Om\cM_g$ for all $g$, and several infinite families of Teichm\"{u}ller curves with degree $2$ trace fields have been discovered \cite{Cal:Teich}, \cite{McM:HMS}, \cite{McM:Prym}, \cite{EMMW:rank2}, \cite{MMW:rank2}. In contrast, in each genus $g \geq 3$, there are only finitely many Teichm\"{u}ller curves whose trace field has degree at least $3$ \cite{EFW:finiteness}. Several other finiteness results are known for such Teichm\"{u}ller curves \cite{BM:real}, \cite{BHM:finiteness}, \cite{MW:fin}, \cite{Mol:finiteness}, \cite{Ham:finiteness}, some of which are in principle effective. Nevertheless, even in genus $3$, a complete classification of primitive Teichm\"{u}ller curves remains a major challenge.

In this paper, we resolve a special case of the classification problem for Teichm\"{u}ller curves in genus $3$. Our focus will be on algebraically primitive Teichm\"{u}ller curves in the hyperelliptic connected component of the stratum $\Om\cM_3(2,2)$. This component consists of pairs $(X,\om)$ where $X$ is hyperelliptic and the hyperelliptic involution $\tau : X \ra X$ exchanges the two zeros of $\om$. There is one known example of an algebraically primitive Teichm\"{u}ller curve in this stratum component, discovered by Veech in \cite{Vee:Teich}, and generated by the holomorphic $1$-form obtained from a regular $14$-gon by identifying pairs of opposite sides. We refer to this holomorphic $1$-form as the {\em Veech $14$-gon}. Our main result is the first complete classification result for algebraically primitive Teichm\"{u}ller curves in a stratum component in genus greater than $2$.

\begin{thm} \label{thm:14gon}
The Veech $14$-gon generates the unique algebraically primitive Teichm\"{u}ller curve in the hyperelliptic component of $\Omega\mc{M}_3(2,2)$.
\end{thm}

Finiteness of algebraically primitive Teichm\"{u}ller curves in this stratum component was first proven in \cite{Mol:finiteness}. Our approach to Theorem \ref{thm:14gon} builds off of the approach in \cite{Mol:finiteness}. In Section \ref{sec:irred}, we gather some useful dynamical, algebro-geometric, and number theoretic results from \cite{Vee:Teich}, \cite{Mol:RM}, \cite{Mol:torsion}, \cite{BM:real}, and we review the approach in \cite{Mol:finiteness} for constraining the irreducible periodic directions of a holomorphic $1$-form generating an algebraically primitive Teichm\"{u}ller curve. We then give a flat-geometric argument showing that there is a unique combinatorial type of irreducible periodic direction in our setting.

Up to a real scalar, the cylinder circumferences in an irreducible periodic direction form a $\Q$-basis for a cubic number field $K$. In \cite{Mol:finiteness}, it is shown that these circumferences are the coefficients of an explicit system of $K$-linear equations in roots of unity. The orders of these roots of unity can be bounded using a generalization of Mann's theorem \cite{Man:roots} in \cite{DZ:roots}. However, the bounds in \cite{DZ:roots} only apply to primitive $K$-linear relations (meaning no proper nonempty subsum vanishes), and the bounds provided are not good enough for a feasible computer search in our setting. In Section \ref{sec:circum}, we apply the results of \cite{DZ:roots}, and some casework based on primitivity and the powers of $2$ dividing the orders of the roots of unity, to narrow down the possible solutions. Once the circumferences are known, a result in \cite{BM:real} determines the corresponding heights, up to a real scalar. Specifically, the circumferences and heights are real multiples of dual bases for $K$ with respect to the trace pairing. Using these constraints and a small amount of computer assistance, we obtain a short list of possible cylinder circumferences and heights.

It remains to determine the possible twist parameters of the horizontal cylinders. In Section \ref{sec:main}, we show that in our setting, it is possible to horizontally shear so that the vertical flow has a first-return map given by a pair of disjoint rotations. These rotations must be periodic \cite{Vee:Teich}, and it is then possible to explicitly write down the moduli of the vertical cylinders and check whether they have rational ratios. This allows us to conclude that the holomorphic $1$-form in question generates the same Teichm\"{u}ller curve as the Veech $14$-gon. \\

\paragraph{\bf Notes and references.} The infinite family of primitive Teichm\"{u}ller curves arising from regular polygons in \cite{Vee:flow} was generalized in \cite{BM:Teich}, and an alternative description of the latter family in terms of semi-regular polygons was discovered in \cite{Hoo:semi}. These families include only finitely many primitive Teichm\"{u}ller curves in each genus. Infinite families of primitive Teichm\"{u}ller curves in genus $2$, $3$, and $4$ were discovered in \cite{Cal:Teich}, \cite{McM:HMS}, \cite{McM:Prym}, \cite{MMW:rank2}, \cite{EMMW:rank2}. In each genus $g \geq 5$, it is unknown whether there are infinitely many primitive Teichm\"{u}ller curves. Primitive Teichm\"{u}ller curves were classified in genus $2$ in \cite{McM:discriminant}, \cite{McM:torsion}. The results of \cite{EFW:finiteness} together with \cite{NW:nonveech}, \cite{ANW:h4odd}, \cite{AN:rank2}, \cite{AN:rank22}, imply that all but finitely many primitive Teichm\"{u}ller curves in genus $3$ have been discovered, as observed in \cite{McM:survey}. This finiteness result is non-effective. Additional non-effective finiteness results appear in \cite{Ham:finiteness}, \cite{MW:fin}, \cite{LNW:fin}, and effective finiteness results appear in \cite{BM:real}, \cite{BHM:finiteness}, \cite{LM:prym}, \cite{Mol:finiteness}. All of these effective finiteness results rely on the special algebro-geometric properties of Teichm\"{u}ller curves established in \cite{Mol:RM}, \cite{Mol:torsion}. \\

\paragraph{\bf Acknowledgements.} The author thanks Matt Bainbridge and Curt McMullen for interesting and informative discussions on topics related to this paper. The author was supported by an NSF GRFP under grant DGE-1144152 and by a Simons Postdoctoral Fellowship at the Fields Institute.


\section{Irreducible periodic directions} \label{sec:irred}

We first review background material on Teichm\"{u}ller curves and Veech surfaces. We then recall several results regarding periodic directions of Veech surfaces and algebro-geometric properties of Teichm\"{u}ller curves, and we review the approach in \cite{Mol:finiteness} for constraining the geometry of an irreducible periodic direction for an algebraically primitive Veech surface in the hyperelliptic component of $\Om\cM_g(g-1,g-1)$. For more background material, we refer to the survey articles \cite{McM:survey}, \cite{Wri:survey}, \cite{Zor:survey}.

Fix $g \geq 2$, and let $\cM_g$ be the moduli space of closed Riemann surfaces of genus $g$. Let $\Om\cM_g$ be the moduli space of pairs $(X,\om)$ where $X \in \cM_g$ and $\om$ is a nonzero holomorphic $1$-form on $X$. Denote by $Z(\om)$ the set of zeros of $\om$. Integrating $\om$ on small open subsets of $X \sm Z(\om)$ gives an atlas of charts to the complex plane $\C$ whose transition maps are translations. The Euclidean metric on $\C$ induces a flat metric $|\om|$ with a cone point of angle $2\pi(k+1)$ at a zero of order $k$. A foliation of $\C$ by parallel lines induces a foliation of $X$ with a $k$-prong singularity at each zero of $\om$ of order $k$.

A {\em saddle connection} on $(X,\om)$ is a geodesic segment for the metric $|\om|$ whose endpoints are zeros of $\om$ and whose interior does not contain any zeros. A {\em cylinder} is a maximal connected open subset of $X \sm Z(\om)$ that is a union of parallel closed geodesics. Any cylinder $C$ is isometric to a unique Euclidean cylinder of the form $\R/c\Z \times (0,h)$ with $c,h > 0$, and we say that $c$ is the {\em circumference} and $h$ the {\em height} of $C$. The {\em modulus} of $C$ is $h/c$. A direction is {\em periodic} for $(X,\om)$ if all of the leaves of the straight-line foliation in that direction are either closed geodesics in $X \sm Z(\om)$ or saddle connections. In this case, $X$ is a union of finitely many parallel cylinders and finitely many parallel saddle connections. A periodic direction is {\em irreducible} if the union of the saddle connections in that direction is connected, and {\em reducible} otherwise.

The space $\Om\cM_g$ is a union of strata $\Om\cM_g(k_1,\dots,k_n)$ consisting of pairs $(X,\om)$ such that $\om$ has exactly $n$ distinct zeros of orders $k_1,\dots,k_n$. Strata can have up to $3$ connected components, which are classified by hyperellipticity and the parity of a spin structure \cite{KZ:components}. In this paper, we will focus on the hyperelliptic component of $\Om\cM_g(g-1,g-1)$, which we denote by $\Om\cM_g(g-1,g-1)^{\rm hyp}$. We have $(X,\om) \in \Om\cM_g(g-1,g-1)^{\rm hyp}$ if and only if $X$ is hyperelliptic and the hyperelliptic involution $\tau : X \ra X$ exchanges the two zeros of $\om$.

Let $\GL^+(2,\R)$ be the group of linear automorphisms of $\R^2$ with positive determinant. The standard $\R$-linear action of $\GL^+(2,\R)$ on the complex plane $\C$ induces an action on $\Om\cM_g$ that preserves connected components of strata. Orbits of the $\GL^+(2,\R)$-action project to complex geodesics in $\cM_g$, and the stabilizer of $(X,\om)$ is a discrete subgroup $\SL(X,\om) \subset \SL(2,\R)$. When $\SL(X,\om)$ is a lattice in $\SL(2,\R)$, the orbit of $(X,\om)$ projects to an isometrically immersed algebraic {\em Teichm\"{u}ller curve} $\H \ra \cM_g$. In this case, $(X,\om)$ is called a {\em Veech surface}. If $\cC$ is the stratum component containing $(X,\om)$, we say that $(X,\om)$ {\em generates} a Teichm\"{u}ller curve in $\cC$. \\

\begin{thm} \cite{Vee:Teich} \label{thm:Veech}
Let $(X,\om)$ be a Veech surface. For any saddle connection on $(X,\om)$, the foliation in the direction of that saddle connection is periodic. Moreover, the moduli of the cylinders on $(X,\om)$ in that direction have rational ratios.
\end{thm}

The {\em trace field} of a Veech surface $(X,\om)$ is the field
\be
K = \Q(\Tr(A) : A \in \SL(X,\om)) .
\ee
In general, $K$ is a number field of degree at most $g$. When equality holds, we say that $(X,\om)$ is {\em algebraically primitive}, and similarly for the Teichm\"{u}ller curve generated by $(X,\om)$.

Teichm\"{u}ller curves are studied from the perspective of algebraic geometry in \cite{Mol:RM} and \cite{Mol:torsion}. It is shown that the associated Jacobians have a factor with extra endomorphisms, and in which the difference of two zeros of the underlying holomorphic $1$-forms is torsion. In the algebraically primitive case, this factor is the entire Jacobian.

\begin{thm} (Theorem 2.7 in \cite{Mol:RM}, Theorem 3.3 in \cite{Mol:torsion}) \label{thm:RMtorsion}
Suppose $(X,\om)$ is an algebraically primitive Veech surface with trace field $K$.
\begin{enumerate}
    \item The $\GL^+(2,\R)$-orbit of $(X,\om)$ lies in the locus of eigenforms for real multiplication by $K$.
    \item The difference of any two zeros of $\om$ is torsion in the Jacobian of $X$.
\end{enumerate}
\end{thm}

In \cite{Mol:finiteness}, degeneration arguments and the real multiplication and torsion conditions in Theorem \ref{thm:RMtorsion} are used to constrain the periodic directions on an algebraically primitive Veech surface $(X,\om)$ in the case where $(X,\om) \in \Om\cM_g(g-1,g-1)^{\rm hyp}$. We briefly review these arguments from \cite{Mol:finiteness} below.

Since $\om$ has exactly $2$ zeros $Z_1,Z_2$, a periodic direction is irreducible if and only if there is a saddle connection in that direction with distinct endpoints. Any holomorphic $1$-form with more than $1$ zero has a saddle connection with distinct endpoints, so by Theorem \ref{thm:Veech}, $(X,\om)$ has an irreducible periodic direction. By rotating, we may assume that the horizontal direction is periodic and irreducible. Since $(X,\omega)$ is algebraically primitive, there are exactly $g$ horizontal cylinders $C_1,\dots,C_g$. Let $c_j$ be the circumference of $C_j$. The hyperelliptic involution $\tau : X \ra X$ fixes each horizontal cylinder, so each horizontal cylinder contains exactly two Weierstrass points in its interior. Since there are $2g+2$ Weierstrass points in total, $\tau$ must fix exactly $2$ horizontal saddle connections, and the remaining $2$ Weierstrass points are the midpoints of these saddle connections. The remaining $2g-2$ horizontal saddle connections are exchanged in pairs by $\tau$.

Now apply the diagonal matrices $a_s = {\rm diag}(1,s) \subset \GL^+(2,\R)$ to $(X,\om)$. As $s \ra +\infty$, the holomorphic $1$-forms $a_s(X,\omega)$ converge to a stable form $\omega_\infty$ on a rational curve $X_\infty$ with $g$ nodes. Geometrically, $(X_\infty,\omega_\infty)$ is obtained from $(X,\omega)$ by removing a closed geodesic from each horizontal cylinder and replacing the resulting open cylinders with half-infinite cylinders. The zeros of $a_s(X,\omega)$ limit to a pair of zeros of $\omega_\infty$ of order $g-1$ at non-singular points on $X_\infty$. For each horizontal cylinder, the $2$ Weierstrass points in the interior of that cylinder collide to form a node on $X_\infty$. The remaining two Weierstrass points limit to a pair of non-singular points on $X_\infty$ distinct from the zeros of $\omega_\infty$. The stable form $\omega_\infty$ has simple poles at the nodes with opposite residues $\pm c_j /2 \pi i$ on the two branches. After lifting to the normalization $\wt{X}_\infty \cong \P^1(\C)$ and applying an automorphism, we can arrange that the zeros of $\wt{\omega}_\infty$ are at $0$ and $\infty$, the hyperelliptic involutions on $a_s(X,\omega)$ limit to $z \mapsto 1/z$, and the two Weierstrass points that remained distinct on $X_\infty$ lift to $\pm 1$. Then the other $2g$ Weierstrass points lift to pairs $x_j^{\pm 1}$ lying above each node. By Theorem \ref{thm:RMtorsion}, the difference of the zeros $Z_2 - Z_1$ on $a_s(X,\om)$ is a torsion divisor, so let $N$ be the torsion order. The torsion order is preserved in the limit (see Lemma 8.3 in \cite{BHM:finiteness}), so there is a meromorphic function on $\wt{X}_\infty$ that factors through $X_\infty$ and has a zero of order $N$ at $0$ and a pole of order $N$ at $\infty$. Up to a constant factor, this function is given by $z \mapsto z^N$. Since $x_j^{\pm 1}$ are identified to form the nodes of $X_\infty$, we must have $x_j^N = x_j^{-N}$, thus the $x_j$ are $2N$-th roots of unity.

Using the residues of $\wt{\om}_\infty$ at its poles, and using the locations of its zeros and poles, we can express the meromorphic $1$-form $\wt{\om}_\infty$ in two ways as
\begin{equation} \label{eq:infty}
\wt{\omega}_\infty = \frac{1}{2 \pi i}\sum_{j=1}^g \left(\frac{c_j}{z - x_j} - \frac{c_j}{z - x_j^{-1}} \right) dz = \frac{A z^{g-1} dz}{\prod_{j=1}^g (z - x_j)(z - x_j^{-1})}
\end{equation}
where $A \in \R$ since $c_j \in \R$. By clearing denominators and comparing coefficients of each power of $z$, after some algebra one arrives an explicit system of equations in the circumferences $c_j$ and the roots of unity $x_j$.

\begin{thm} (Section 3 and (Eq:$e^{\pr\pr}$) in \cite{Mol:finiteness}) \label{thm:res}
Let $(X,\omega) \in \Omega\mc{M}_g(g-1,g-1)^{\rm hyp}$ be an algebraically primitive Veech surface with trace field $K$ such that the horizontal direction is periodic and irreducible, and let $c_1,\dots,c_g$ be the circumferences of the $g$ horizontal cylinders. There exist roots of unity $x_1,\dots,x_g \in \C$ satisfying
\begin{align*}
\sum_{j=1}^g c_j (x_j^r - x_j^{-r}) = 0, \quad r = 1,\dots,g-1 .
\end{align*}
\noindent
Moreover, the roots of unity $\pm 1, x_1^{\pm 1}, \dots, x_g^{\pm 1}$ are distinct, and the tuple of circumferences $(c_1, \dots, c_g)$ is a real multiple of a $\Q$-basis for $K \subset \Q(x_1,\dots,x_g)$.
\end{thm}

The containment $K \subset \Q(x_1,\dots,x_g)$ follows from the invertibility of the matrix $M_{ij} = x_j^{i-1} - x_j^{-i+1}$, $1 \leq i,j \leq g$, since $\pm 1, x_1^{\pm 1}, \dots, x_g^{\pm 1}$ are distinct. In particular, the circumferences $c_j$ are determined by their sum and by the $x_j$. The circumferences $c_j$ form a $\Q$-basis for $K$ since $(X,\om)$ is algebraically primitive. Another consequence of algebraic primitivity is that the tuple of heights is determined up to a real scalar by the tuple of circumferences. The {\em trace pairing} $K \times K \ra \Q$ is defined by $(a,b) \mapsto \Tr_{K/\Q}(ab)$.

\begin{thm} (Lemma 10.4 in \cite{BM:real}) \label{thm:ratpos}
Let $(X,\omega) \in \Omega\mc{M}_g$ be an algebraically primitive Veech surface with trace field $K$ such that the horizontal direction is periodic and irreducible. Let $c_1,\dots,c_g$ be the circumferences of the $g$ horizontal cylinders, and let $h_1,\dots,h_g$ be the corresponding heights. Then $(c_1,\dots,c_g)$ is a real multiple of a $\Q$-basis for $K$, and $(h_1,\dots,h_g)$ is a real multiple of the dual basis with respect to the trace pairing. Moreover, the ratios $h_j c_k / c_j h_k$ are rational and positive.
\end{thm}

The last statement is part of Theorem \ref{thm:Veech}. Circumferences arising from solutions to the equations in Theorem \ref{thm:res} do not necessarily satisfy the constraints in Theorem \ref{thm:ratpos}, and in practice this rules out the vast majority of solutions from potentially arising from algebraically primitive Veech surfaces.

We now present a flat-geometric argument that may be of independent interest, which uniquely determines the combinatorial type of the cylinder decomposition in an irreducible periodic direction of an algebraically primitive Veech surface in $\Om\cM_g(g-1,g-1)^{\rm hyp}$. The argument uses the equation for the stable form in (\ref{eq:infty}) to produce an orientation-reversing symmetry of the associated Veech surfaces.

The {\em cylinder digraph} $\Gam(X,\om)$ is the directed graph with a vertex for each horizontal cylinder $C$, and an edge from $C_1$ to $C_2$ for each horizontal saddle connection in the intersection of the top boundary of $C_1$ with the bottom boundary of $C_2$. We say that a directed graph is a {\em chain} if the underlying undirected graph is connected and each vertex has at most $2$ adjacent vertices. (There may be multiple edges between two adjacent vertices.) In the case where $(X,\om) \in \Om\cM_g(g-1,g-1)^{\rm hyp}$, the hyperelliptic involution $\tau$ fixes each horizontal cylinder $C$ and exchanges the top and bottom boundaries of $C$. Thus, if $C_1 \neq C_2$, the number of horizontal saddle connections in $\ol{C_1} \cap \ol{C_2}$ is even, and these saddle connections are exchanged in pairs by $\tau$. Also, recall from the discussion below Theorem \ref{thm:RMtorsion} that there are exactly $2$ horizontal saddle connections fixed by $\tau$.

\begin{lem} \label{lem:hypirred}
Let $(X,\omega) \in \Omega\mc{M}_g (g-1,g-1)^{\rm hyp}$ be an algebraically primitive Veech surface such that the horizontal direction is periodic and irreducible. Then each horizontal cylinder has exactly $2$ saddle connections in its top boundary and exactly $2$ saddle connections in its bottom boundary. Moreover, the cylinder digraph $\Gam(X,\om)$ is a chain.
\end{lem}

\begin{proof}
Following the discussion below Theorem \ref{thm:RMtorsion}, by applying the diagonal matrices $a_s$ to $(X,\om)$ with $s \ra +\infty$ and normalizing, we obtain a meromorphic $1$-form
\be
\wt{\om}_\infty = \frac{1}{2\pi i}\sum_{j=1}^g \left(\frac{c_j}{z-x_j} - \frac{c_j}{z-x_j^{-1}}\right) dz
\ee
on $\P^1(\C)$ with zeros of order $g-1$ at $0$ and $\infty$, and with simple poles at roots of unity $x_1^{\pm 1}, \dots, x_g^{\pm 1}$, such that $\pm 1, x_1^{\pm 1}, \dots, x_g^{\pm 1}$ are distinct.

Since $c_j \in \R$ and $x_j^{-1} = \ol{x_j}$, complex conjugation $z \mapsto \ol{z}$ defines an orientation-reversing isometry on $\P^1(\C) \sm \{x_1^{\pm 1},\dots,x_g^{\pm 1}\}$ with respect to the metric $|\wt{\omega}_\infty|$. This isometry fixes the zeros $0,\infty$, exchanges the poles in pairs $x_j^{\pm 1}$, and fixes the points $\pm 1$. Each pair of poles $x_j^{\pm 1}$ has a neighborhood contained in a pair of half-infinite horizontal cylinders with the same circumference $c_j$ arising from a single horizontal cylinder $C_j$ on $(X,\om)$. The isometry $z \mapsto \ol{z}$ exchanges these horizontal cylinders, and therefore fixes the horizontal direction for $\wt{\om}_\infty$. Then the set of fixed points $\R \cup \infty$ is a union of horizontal saddle connections for $\wt{\om}_\infty$. Recall that $\wt{\om}_\infty$ is obtained from $(X,\om)$ by removing a closed geodesic from each horizontal cylinder and replacing the resulting open cylinders with half-infinite horizontal cylinders. In particular, the horizontal saddle connections and their lengths are preserved in this process. The two Weierstrass points on $X$ that limit to $\pm 1 \in \P^1(\C)$ lie on two distinct horizontal saddle connections $\gam_0,\gam_g$. Letting $\wt{\gam}_0,\wt{\gam}_g$ be the corresponding horizontal saddle connections for $\wt{\om}_\infty$, we may assume that $1 \in \wt{\gam}_0$ and $-1 \in \wt{\gam}_g$. Since the two zeros of $\wt{\om}_\infty$ are $0,\infty$, the interior of $\wt{\gam}_0$ is the positive part of the real line, and the interior of $\wt{\gam}_g$ is the negative part of the real line.

Passing back to $(X,\om)$, there is an orientation-reversing isometry $f : X \ra X$ with respect to $|\om|$ that fixes each horizontal cylinder while exchanging its top and bottom boundaries, fixes the zeros of $\om$, and fixes the saddle connections $\gam_0$ and $\gam_g$ pointwise. Recall that the hyperelliptic involution $\tau : X \ra X$ fixes each horizontal cylinder while exchanging its top and bottom boundaries, exchanges the two zeros of $\om$, and fixes the saddle connections $\gam_0$ and $\gam_g$. Moreover, $\gam_0$ and $\gam_g$ are the only horizontal saddle connections fixed by $\tau$.

Suppose that $(X,\om)$ has a horizontal cylinder $C$ with only one saddle connection in one of its boundaries. Then by applying $\tau$, we see that each of the top and bottom boundaries of $C$ contains only one saddle connection. Since $\tau$ exchanges the two zeros, these two saddle connections are loops through distinct zeros. Then $f$ also exchanges these two saddle connections, a contradiction since $f$ fixes the zeros of $\om$. Thus, each horizontal cylinder contains at least $2$ saddle connections in each of its top and bottom boundaries. Since each saddle connection is contained in the top boundary of a unique horizontal cylinder and the bottom boundary of a unique horizontal cylinder, we have accounted for at least $2g$ horizontal saddle connections. Since $\om$ has exactly two zeros, each of order $g-1$, there are exactly $2g$ horizontal saddle connections on $(X,\om)$. Thus, each horizontal cylinder contains exactly $2$ saddle connections in its top boundary and exactly $2$ saddle connections in its bottom boundary.

Let $C_1$ be the horizontal cylinder containing $\gam_0$ in both its top and bottom boundaries, and let $C_g$ be the cylinder containing $\gam_g$ in both its top and bottom boundaries. Then $C_1$ is adjacent to exactly one other horizontal cylinder $C_2$, and $\ol{C_1} \cap \ol{C_2}$ consists of exactly $2$ horizontal saddle connections $\gam_1,\gam_1^\pr$ exchanged by $\tau$ and $f$. Then similarly, $C_2$ is adjacent to exactly one horizontal cylinder $C_3$ other than $C_1$, and $\ol{C_2} \cap \ol{C_3}$ consists of exactly $2$ horizontal saddle connections $\gam_2,\gam_2^\pr$ exchanged by $\tau$ and $f$. Continuing in this way, $\ol{C_j} \cap \ol{C_{j+1}}$ consists of exactly $2$ horizontal saddle connections $\gam_j,\gam_j^\pr$ for $j = 1,\dots,g-1$. Thus, the cylinder digraph $\Gam(X,\om)$ is a chain with endpoints $C_1$ and $C_g$.
\end{proof}

\begin{rmk} \label{rmk:distinct}
Lemma \ref{lem:hypirred} implies that for any saddle connection on $(X,\om)$ with distinct endpoints, every saddle connection in that direction has distinct endpoints. Thus, for any saddle connection on $(X,\om)$ that is a loop, every saddle connection in that direction is also a loop.
\end{rmk}

\begin{rmk} \label{rmk:gap}
Lemma \ref{lem:hypirred} provides a flat-geometric perspective on the prototype described in Figure 3.1 in \cite{Mol:finiteness}.
\end{rmk}

Lastly, in the notation of the proof of Lemma \ref{lem:hypirred}, suppose the horizontal cylinder circumferences $c_1,\dots,c_g$ have been fixed, and orient each horizontal saddle connection from left to right. The length of the distinguished horizontal saddle connection $\gam_0$ determines the length of every other horizontal saddle connection and is constrained by the fact that these lengths are positive. Additionally, since $\int_{\gam_0} \om = \int_{\wt{\gam_0}}\wt{\om}_\infty$, by using the expression in (\ref{eq:infty}) we can compute $\int_{\gam_0} \om$ in terms of the circumferences $c_j$ and the roots of unity $x_j$. We record the result of these computations below. For a nonzero real number $t$, let $\sgn(t) = 1$ if $t > 0$ and $\sgn(t) = -1$ if $t < 0$.

\begin{lem} \label{lem:sc}
With notation as in the proof Lemma \ref{lem:hypirred}, the length of the saddle connection $\gam_0$ satisfies $\int_{\gam_0} \om > 0$ and
\be
(-1)^k \int_{\gam_0} \om + \sum_{j=1}^k (-1)^{j-1} c_j > 0
\ee
for $k = 1,\dots,g-1$. Writing $x_j = \zeta_{2N}^{n_j}$, where $\zeta_{2N} = \exp(\pi i / N)$ and $-N < n_j < N$, we have
\be
\int_{\gam_0} \om = \sum_{j=1}^g c_j \left(\sgn(n_j) - \frac{n_j}{N}\right) .
\ee
\end{lem}


\section{Cylinders in Irreducible Directions} \label{sec:circum}

In this section, we use the constraints from Section \ref{sec:irred}, together with a generalization of Mann's theorem \cite{Man:roots} in \cite{DZ:roots}, to produce a short list of possible cylinder circumferences in an irreducible periodic direction for an algebraically primitive Veech surface in $\Om\cM_3(2,2)^{\rm hyp}$. The following theorem summarizes the results from Section \ref{sec:irred}.

\begin{thm} \label{thm:irred}
Suppose that $(X_0,\om_0) \in \Om\cM_g(g-1,g-1)^{\rm hyp}$ is an algebraically primitive Veech surface. There is $(X,\om) \in \GL^+(2,\R) \cdot (X_0,\om_0)$, such that the horizontal direction is periodic and irreducible, with the following properties.
\begin{enumerate}
    \item There are $g$ horizontal cylinders $C_1,\dots,C_g$ on $(X,\om)$, and the cylinder digraph $\Gam(X,\om)$ is a chain with endpoints $C_1,C_g$.
    \item The circumferences $c_1,\dots,c_g$ and heights $h_1,\dots,h_g$ of $C_1,\dots,C_g$ are $\Q$-bases for a number field $K \subset \Q(x_1,\dots,x_g)$ that are dual with respect to the trace pairing, and the moduli $m_1,\dots,m_g$ have rational ratios.
    \item There exist roots of unity $x_1,\dots,x_g$ such that $\pm 1, x_1^{\pm 1}, \dots, x_g^{\pm g}$ are distinct and
    \be
    \sum_{j=1}^g c_j (x_j^r - x_j^{-r}) = 0, \quad r = 1,\dots,g-1.
    \ee
    \item Writing $x_j = \zeta_{2N}^{n_j}$, where $\zeta_{2N} = \exp(\pi i / N)$ and $-N < n_j < N$, and letting
    \be
    s = \sum_{j=1}^g c_j \left(\sgn(n_j) - \frac{n_j}{N}\right) ,
    \ee
    we have $s > 0$ and
    \be
    (-1)^k s + \sum_{j=1}^k (-1)^{j-1} c_j > 0
    \ee
    for $k = 1,\dots,g$.
\end{enumerate}
\end{thm}

For the rest of this section, we will focus on the case $g = 3$. We need to find all solutions to the equations
\begin{align}
    c_1 (x_1 - x_1^{-1}) + c_2 (x_2 - x_2^{-1}) + c_3 (x_3 - x_3^{-1}) &= 0 \label{eq:res1} \\
    c_1 (x_1^2 - x_1^{-2}) + c_2 (x_2^2 - x_2^{-2}) + c_3 (x_3^2 - x_3^{-2}) &= 0 \label{eq:res2}
\end{align}
where $c_1,c_2,c_3,x_1,x_2,x_3$ satisfy the conditions in Theorem \ref{thm:irred}. We will call such a solution to equations (\ref{eq:res1}) and (\ref{eq:res2}) an {\em admissible} solution.

Since $\pm1,x_1^{\pm1},x_2^{\pm1},x_3^{\pm1}$ are distinct, we have
\be
(x_j - x_j^{-1}) (x_{j+1}^2 - x_{j+1}^{-2}) - (x_j^2 - x_j^{-2}) (x_{j+1} - x_{j+1}^{-1}) \neq 0
\ee
\noindent
for $j \in \{1,2,3\}$, where indices are taken modulo $3$. Thus, we can write $c_j/c_k$ as
\begin{equation} \label{eq:cjck}
\frac{c_j}{c_k} = \frac{(x_{j+1} - x_{j+1}^{-1})(x_{j+2}^2 - x_{j+2}^{-2}) - (x_{j+1}^2 - x_{j+1}^{-2})(x_{j+2} - x_{j+2}^{-1})}{(x_{k+1} - x_{k+1}^{-1})(x_{k+2}^2 - x_{k+2}^{-2}) - (x_{k+1}^2 - x_{k+1}^{-2})(x_{k+2} - x_{k+2}^{-1})} \in K \setminus \Q .
\end{equation}

A number field $L$ is {\em abelian} if it is a subfield of $\Q(\zeta)$ for some root of unity $\zeta$. Any subfield of $\Q(\zeta)$ has an abelian Galois group. The {\em conductor} $f(L)$ is the smallest positive integer $n$ such that $L$ is a subfield of $\Q(\zeta_n)$, where $\zeta_n$ is a primitive $n$th root of unity. The conductor of a cubic abelian number field $L$ has the form
\begin{equation} \label{eq:cond}
f(L) = 9^{\eps} p_1 \cdots p_m
\end{equation}
\noindent
for some $\eps \in \{0,1\}$ and distinct primes $p_j$ such that $p_j \equiv 1 \pmod 3$, and there are $2^{\eps + m}$ cubic abelian number fields with conductor $9^{\eps} p_1 \cdots p_m$ (see, for instance, \cite{May:cond}).

Our main tool for constraining the admissible solutions of equations (\ref{eq:res1}) and (\ref{eq:res2}) is a generalization of Mann's theorem \cite{Man:roots} on $\Q$-linear relations between roots of unity. Let $L$ be a number field. An {\em $L$-relation} is a linear relation $\sum_{j=1}^k a_j z_j = 0$ among distinct roots of unity $z_j$ with coefficients $a_j \in L$. An $L$-relation is {\em primitive} if $\sum_{j \in S} a_j z_j \neq 0$ for all proper nonempty subsets $S \subset \{1,\dots,k\}$. Any $L$-relation is a sum of primitive $L$-relations, possibly in more than one way. The {\em length} of an $L$-relation is the number of terms $k$, and the {\em order} of an $L$-relation is the size of the multiplicative group generated by the roots of unity $z_j / z_1$. The notation $a \mid b$ means that $a$ divides $b$. The notation $a \left|\right| b$ means that $a$ divides $b$ exactly once, that is, $b/a$ is an integer but $b/a^2$ is not. The following theorem provides an effective bound on the order of a primitive $L$-relation in terms of its length and the degree of $L$.

\begin{thm} \label{thm:DZ} (Theorem 1 in \cite{DZ:roots})
Suppose that $\sum_{j=1}^k a_j z_j = 0$ is a primitive $L$-relation of length $k$ and order $n$. Let $\zeta_n$ be a primitive $n$-th root of unity, and let $d = [L \cap \Q(\zeta_n) : \Q]$. For any prime $p$ and any positive integer $m$ such that $p^{m+1} \mid n$, we have $p^m \mid 2d$. Moreover, the primes that divide $n$ exactly once are bounded by
\be
\sum_{p \left|\right| n} \left(\frac{p-1}{\gcd(d,p-1)} - 1\right) \leq k - 2 .
\ee
\end{thm}

We want to apply Theorem \ref{thm:DZ} to bound the orders of the roots of unity in an admissible solution to equations (\ref{eq:res1}) and (\ref{eq:res2}). However, this requires some casework, since the relations in these equations are not necessarily primitive. Ultimately, we will obtain the following bound.

\begin{thm} \label{thm:bound}
Suppose that $c_1,c_2,c_3,x_1,x_2,x_3$ are an admissible solution to equations to (\ref{eq:res1}) and (\ref{eq:res2}). Let $K = \Q(c_1,c_2,c_3)$, and let $n = \lcm(\ord(x_1),\ord(x_2),\ord(x_3))$. Then either $n \in \{7,14\}$ and $K = \Q(\cos(\pi/7))$, or $n = 18$ and $K = \Q(\cos(\pi/9))$.
\end{thm}

We keep the notation from Theorem \ref{thm:bound} for the rest of this section. First, we obtain a weaker bound on $n$.

\begin{lem} \label{lem:bound}
We have $n \mid 2^4 \cdot 3^2 \cdot 7 \cdot 13$.
\end{lem}

\begin{proof}
Let $n^\pr = \lcm(\ord(x_j x_k) : j,k \in \{1,2,3\})$. If the $K$-relation in (\ref{eq:res1}) is primitive, then Theorem \ref{thm:DZ} with $k = 6$ and $d \mid 3$ implies $n^\pr \mid 2^2 \cdot 3^2 \cdot 5 \cdot 7$ or $n^\pr \mid 2^2 \cdot 3^2 \cdot 7 \cdot 13$. In particular, since $\ord(x_j^2) \mid n^\pr$ for $j \in \{1,2,3\}$, we have $n \mid 2^3 \cdot 3^2 \cdot 5 \cdot 7$ or $n \mid 2^3 \cdot 3^2 \cdot 7 \cdot 13$. Similarly, if the $K$-relation in (\ref{eq:res2}) is primitive, then $n \mid 2^4 \cdot 3^2 \cdot 5 \cdot 7$ or $n \mid 2^4 \cdot 3^2 \cdot 7 \cdot 13$. We will rule out the factor of $5$ at the end.

Suppose that neither $K$-relation is primitive. Then each of these $K$-relations contains a primitive $K$-relation of length $2$ or $3$. In the length $2$ case, we have
\be
\eps_1 c_j x_j^{\eps_1 r} + \eps_2 c_k x_k^{\eps_2 r} = 0
\ee
for some $\eps_1,\eps_2 \in \{\pm 1\}$, $j,k \in \{1,2,3\}$, $r \in \{1,2\}$. If $j \neq k$, then $c_j/c_k$ is a real root of unity and so $c_j/c_k = \pm 1$. This contradicts the $\Q$-linear independence of $c_1,c_2,c_3$. If $j = k$, then $\eps_2 = -\eps_1$ and $x_j^r - x_j^{-r} = 0$, and so $x^{2r} = 1$. Since $\pm 1, x_j^{\pm 1}$ are distinct, it must be that $r = 2$ and $x_j \in \{\pm i\}$. After permuting indices, we have $c_1 (x_1^2 - x_1^{-2}) + c_2 (x_2^2 - x_2^{-2}) = 0$. If this $K$-relation is primitive, then Theorem \ref{thm:DZ} with $k = 4$ and $d \mid 3$ implies $n \mid 2^4 \cdot 3^2 \cdot 7$, and otherwise, as above it must be that $x_j \in \{\pm i\}$ for $j \in \{1,2,3\}$.

Next, suppose that each $K$-relation contains a primitive $K$-relation of length $3$. If some $x_j$ appears twice in a primitive $K$-relation of length $3$, then after permuting indices, we have
\be
c_1 (x_1^r - x_1^{-r}) + \eps c_2 x_2^{\eps r} = 0, \quad c_3 (x_3^r - x_3^{-r}) - \eps c_2 x_2^{-\eps r} = 0
\ee
for some $\eps \in \{\pm 1\}$, $r \in \{1,2\}$. Then Theorem \ref{thm:DZ} with $k = 3$ and $d \mid 3$ implies $n \mid 2^4 \cdot 3^2 \cdot 7$. Otherwise, each $x_j$ appears exactly once in each primitive $K$-relation of length $3$, and we have
\be
c_1 x_1^r + c_2 x_2^r + c_3 x_3^r = 0
\ee
for $r \in \{\pm 1, \pm 2\}$. Since $x_1,x_2,x_3$ are distinct, we can solve for $c_j/(c_1 + c_2 + c_3)$ for $j \in \{1,2,3\}$ in two different ways, using the equations for $r \in \{1,2\}$ and the equations for $r \in \{-1,-2\}$, to get
\be
\frac{c_j}{c_1+c_2+c_3} = \frac{-x_{j+1}x_{j+2}}{(x_j-x_{j+1})(x_{j+2}-x_j)} = \frac{-x_{j+1}^{-1}x_{j+2}^{-1}}{(x_j^{-1}-x_{j+1}^{-1})(x_{j+2}^{-1}-x_j^{-1})}
\ee
where indices are taken modulo $3$. After simplifying, we get $x_j^2 = x_{j+1}x_{j+2}$, and after permuting indices, $x_2 = \zeta_3 x_1$ and $x_3 = \zeta_3^2 x_1$ where $\zeta_3$ is a primitive third root of unity. Then the relations $c_1 + c_2 \zeta_3 + c_3 \zeta_3^2 = 0$ and $c_1 + c_2 \zeta_3^2 + c_3 \zeta_3 = 0$ imply $c_1 = c_2 = c_3$, contradicting the $\Q$-linear independence of $c_1,c_2,c_3$.

It remains to show that $5 \nmid \ord(x_j)$ for $j \in \{1,2,3\}$. By the previous two paragraphs, we may assume that the $K$-relations in equations (\ref{eq:res1}) and (\ref{eq:res2}) are primitive. Then we can write $x_j = \zeta_5^{n_j} y_j$ with $\ord(y_j) \mid 2^4 \cdot 3^2 \cdot 7 \cdot 13$, where $\zeta_5$ is a primitive fifth root of unity. Since $K \subset \Q(x_1,x_2,x_3)$, and since the conductor $f(K)$ is not divisible by $5$, we have $K \subset \Q(y_1,y_2,y_3)$. Since $5 \nmid \ord(y_j)$ for $j \in \{1,2,3\}$, the number fields $\Q(\zeta_5)$ and $\Q(y_1,y_2,y_3)$ are linearly disjoint over $\Q$. Suppose that $n_1,n_2,n_3 \in \Z/5$ are not all zero. If $S = \{\pm n_1, \pm n_2, \pm n_3\}$ is a proper subset of $\Z/5$, then the powers $\{\zeta_5^m : m \in S\}$ are linearly independent over $\Q$, and therefore linearly independent over $\Q(y_1,y_2,y_3)$. Then by grouping terms in equation (\ref{eq:res1}) according to their power of $\zeta_5$, we see that the $K$-relation in (\ref{eq:res1}) is not primitive, a contradiction. If $S = \Z/5$, then after permuting indices and replacing some of the $x_j$ with $-x_j^{-1}$, we have $n_1 = 0$, $n_2 = 1$, $n_3 = 2$. Since $\zeta_5^4 = -(1 + \zeta_5 + \zeta_5^2 + \zeta_5^3)$, we can rewrite equation (\ref{eq:res1}) as
\be
(c_1y_1 - c_1y_1^{-1} + c_2y_2^{-1}) + \zeta_5(c_2y_2 + c_2y_2^{-1}) + \zeta_5^2(c_2y_2^{-1} + c_3y_3) + \zeta_5^3(c_2y_2^{-1} - c_3y_3^{-1}) = 0 .
\ee
Then each of the $4$ terms must be equal to $0$. By considering the $\zeta_5^2$ term, we see that $c_2/c_3 = -y_3/y_2$ is a real root of unity, so $c_2/c_3 = \pm 1$, contradicting the $\Q$-linear independence of $c_1,c_2,c_3$.
\end{proof}

The bound in Lemma \ref{lem:bound} is not good enough for a feasible computer search, so we will narrow down the possible solutions by considering cases according to the powers of $2$ that divide $\ord(x_j)$.

\begin{lem} \label{lem:2pow4}
If $2^m \mid \ord(x_j)$ for some $m \geq 2$ and some $j \in \{1,2,3\}$, then $m \leq 3$ and $2^m \mid \ord(x_j)$ for all $j \in \{1,2,3\}$.
\end{lem}

\begin{proof}
We may assume that $m$ is minimal such that we can write $x_j = \zeta_{2^m}^{n_j}y_j$ with $\ord(y_j)$ odd for all $j$, where $\zeta_{2^m}$ is a primitive $2^m$-th root of unity and $n_j \in \Z/2^m$. The number fields $\Q(\zeta_{2^m})$ and $\Q(y_1,y_2,y_3)$ are linearly disjoint over $\Q$. Since $K \subset \Q(x_1,x_2,x_3)$ and the conductor $f(K)$ is odd, we have $K \subset \Q(y_1,y_2,y_3)$. The powers $1,\zeta_{2^m},\dots,\zeta_{2^m}^{2^{m-1}-1}$ are linearly independent over $\Q$, and therefore linearly independent over $\Q(y_1,y_2,y_3)$, and $\zeta_{2^m}^k = -\zeta_{2^m}^{k-2^{m-1}}$. Suppose $m \geq 3$, and group the terms in equation (\ref{eq:res1}) according to their power of $\zeta_{2^m}$. For $j,k \in \{1,2,3\}$, the subsets $\{\pm n_j, \pm n_j + 2^{m-1}\}$ and $\{\pm n_k, \pm n_k + 2^{m-1}\}$ of $\Z/2^m$ are either equal or disjoint. If they are not always equal, then after permuting indices, $\{\pm n_1, \pm n_1 + 2^{m-1}\}$ is disjoint from $\{\pm n_j, \pm n_j + 2^{m-1}\}$ for $j \in \{2,3\}$, which implies the $K$-relation in (\ref{eq:res1}) is not primitive, and in particular $\zeta_{2^m}^{n_1}y_1 - \zeta_{2^m}^{-n_1}y_1^{-1} = x_1 - x_1^{-1} = 0$, contradicting that $\pm 1, x_1^{\pm 1}$ are distinct. Thus, the subsets $\{\pm n_j, \pm n_j + 2^{m-1}\}$ are all equal, and the subsets $\{\pm 2n_j\}$ are all equal. This shows that $2^m \mid \ord(x_j)$ for all $j$. Suppose $m = 2$. Then either $2^2 \mid \ord(x_j)$ for all $j$, or by grouping the terms in equation (\ref{eq:res1}) according to their power of $\zeta_{2^2}$, after permuting indices we get that $x_1 - x_1^{-1} = 0$, contradicting that $\pm 1, x_1^{\pm 1}$ are distinct.

It remains to show that $m \leq 3$. Suppose that $m \geq 4$. Since the subsets $\{\pm n_j, \pm n_j + 2^{m-1}\}$ are all equal, after replacing some of the $c_j$ with $-c_j$, and replacing some of the $y_j$ with $y_j^{-1}$, we can rewrite equations (\ref{eq:res1}) and (\ref{eq:res2}) as
\begin{align*}
\zeta_{2^m}^{n_1}(c_1y_1 + c_2y_2 + c_3y_3) - \zeta_{2^m}^{-n_1}(c_1y_1^{-1} + c_2y_2^{-1} + c_3y_3^{-1}) &= 0 \\
\zeta_{2^m}^{2n_1}(c_1y_1^2 + c_2y_2^2 + c_3y_3^2) - \zeta_{2^m}^{-2n_1}(c_1y_1^{-2} + c_2y_2^{-2} + c_3y_3^{-2}) &= 0 .
\end{align*}
Since $n_1$ is odd and $m \geq 4$, $\zeta_{2^m}^{\pm n_1}$ and $\zeta_{2^m}^{\pm 2n_1}$ are linearly independent over $\Q(y_1,y_2,y_3)$, so $c_1 y_1^r + c_2 y_2^r + c_3 y_2^r = 0$ for $r \in \{\pm 1, \pm 2\}$. If $y_1,y_2,y_3$ are distinct, then by the proof of Lemma \ref{lem:bound}, we see that $c_1,c_2,c_3$ are not $\Q$-linearly independent, a contradiction. Otherwise, after permuting indices, we have $y_2 = y_3$ and $c_1 y_1 + (c_2 + c_3) y_2 = 0$. Then $-(c_2 + c_3)/c_1 = y_1/y_2$ is a real root of unity, so $c_2 + c_3 = \pm c_1$, again contradicting the $\Q$-linear independence of $c_1,c_2,c_3$.
\end{proof}

In light of lemma \ref{lem:2pow4}, we are left with $3$ cases.

\begin{lem} \label{lem:2pow3}
If $2^3 \mid \ord(x_j)$ and $2^4 \nmid \ord(x_j)$ for all $j \in \{1,2,3\}$, then $n \in \{2^3 \cdot 7, 2^3 \cdot 9\}$.
\end{lem}

\begin{proof}
Write $x_j = \zeta_{2^3}^{n_j}y_j$ with $\ord(y_j)$ odd. By the proof of Lemma \ref{lem:2pow4}, after replacing some of the $x_j$ with $x_j^{-1}$, and replacing some of the $c_j$ with $-c_j$, we have $n_1 = n_2 = n_3 \in \Z/2^3$ and $c_1 x_1^r + c_2 x_2^r + c_3 x_3^r = 0$ for $r \in \{\pm 1\}$. Since $\zeta_{2^3}^{2n_1} \in \{\pm i\}$, we can rewrite equation (\ref{eq:res2}) as
\begin{equation} \label{eq:2pow3}
c_1 (y_1^2 + y_1^{-2}) + c_2 (y_2^2 + y_2^{-2}) + c_3 (y_3^2 + y_3^{-2}) = 0 .
\end{equation}

There is a field automorphism $\sig$ of $\Q(x_1,x_2,x_3)$ such that $\sig(\zeta_{2^3}) = \zeta_{2^3}$ and $\sig(y_j) = y_j^2$ for all $j$. Suppose that $K$ is contained in the fixed field of $\sig$. By applying $\sig$ and $\sig^2$ to the relation $c_1 x_1 + c_2 x_2 + c_3 x_3 = 0$, and then multiplying by appropriate powers of $\zeta_{2^3}$, we get
\be
c_1 x_1^{2^m} + c_2 x_2^{2^m} + c_3 x_3^{2^m} = 0
\ee
for $m \in \{0,1,2\}$. Since $c_1,c_2,c_3$ are nonzero, the matrix $M_{jk} = x_j^{2^{k-1}}$, $j,k \in \{1,2,3\}$, has vanishing determinant, which gives us
\be
0 = \det(M) = x_1 x_2 x_3 (x_1 - x_2) (x_2 - x_3) (x_3 - x_1) (x_1 + x_2 + x_3) .
\ee
Since $x_1,x_2,x_3$ are distinct and nonzero, it must be that $x_1 + x_2 + x_3 = 0$. By applying a field automorphism of $\Q(x_1,x_2,x_3)$ to the relation $1 + x_2/x_1 = -x_3/x_1$, we see that $1 + e^{2\pi i / m}$ is a root of unity, where $m = \ord(x_2/x_1)$. Since $m > 2$, it must be that $m = 3$, so after permuting indices, $x_2 = \zeta_3 x_1$ and $x_3 = \zeta_3^{-1} x_1$. Then the relations $c_1 + c_2\zeta_3 + c_3\zeta_3^{-1} = 0$ and $c_1 + c_2\zeta_3^{-1} + c_3\zeta_3 = 0$ imply $c_1 = c_2 = c_3$, contradicting the $\Q$-linear independence of $c_1,c_2,c_3$. Thus, $K$ is not contained in the fixed field of $\sig$.

Since $K$ has prime degree over $\Q$, the intersection of the fixed field of $\sig$ with $K$ must be $\Q$. Replace $c_j$ with $c_j/c_1$ for all $j$, so that $c_2,c_3 \notin \Q$. By applying $\sig$ to the relation $y_1^r + c_2y_2^r + c_3y_3^r = 0$ for $r \in \{\pm 1\}$, we get
\be
y_1^2 + \sig(c_2)y_2^2 + \sig(c_3)y_3^2 = 0, \quad y_1^{-2} + \sig(c_2)y_2^{-2} + \sig(c_3)y_3^{-2} = 0 .
\ee
and subtracting from equation (\ref{eq:2pow3}) gives us
\be
(c_2 - \sig(c_2))(y_2^2 + y_2^{-2}) + (c_3 - \sig(c_3))(y_3^2 + y_3^{-2}) = 0 .
\ee
We can then apply $\sig^{-1}$ to get an equation in $y_2^{\pm 1}, y_3^{\pm 1}$. We can carry out this procedure for each pair $j \neq k$ in $\{1,2,3\}$ to get
\begin{equation} \label{eq:cjk}
(y_j + y_j^{-1}) + c_{jk} (y_k + y_k^{-1}) = 0
\end{equation}
for some nonzero $c_{jk} \in K$. Note that since $x_1^{\pm 1}, x_2^{\pm 2}, x_3^{\pm 3}$ are distinct and $n_1 = n_2 = n_3$, $y_1^{\pm 1}, y_2^{\pm 1}, y_3^{\pm 1}$ are distinct. Since $\ord(y_j)$ is odd for all $j$, the $K$-relations in (\ref{eq:cjk}) must be primitive. Then Theorem \ref{thm:DZ} with $k = 4$ and $d \mid 3$ implies $\ord(y_j) \mid 3^2 \cdot 7$ for all $j$.

We use a computer search to find all tuples $(y_1,y_2)$ with $y_1 \neq y_2^{\pm 1}$, $\ord(y_1) \mid 3^2 \cdot 7$ and $\ord(y_2) \mid 3^2 \cdot 7$, such that $(y_1 + y_1^{-1})/(y_2 + y_2^{-1})$ is either rational or cubic. We find that for all such tuples, we have $\gcd(\ord(y_1),\ord(y_2)) \in \{7,9\}$, or after permuting indices $\ord(y_1) = 7$ and $\ord(y_2) = 3$. In the latter case, by considering the tuples $(y_2,y_3)$ and $(y_1,y_3)$ we see that $\ord(y_1) = \ord(y_3) = 7$. The relations $y_1^r + c_2y_2^r + c_3y_3^r = 0$ for $r \in \{\pm 1\}$ then imply
\be
1 = (c_2 y_2 + c_3 y_3)(c_2 y_2^{-1} + c_3 y_3^{-1}) = c_2^2 + c_3^2 + c_2 c_3 (y_2 y_3^{-1} + y_2^{-1} y_3) .
\ee
Then $y_2 y_3^{-1} + y_2^{-1} y_3 \in K$, but $\ord(y_2 y_3^{-1}) = 3 \cdot 7$ implies $y_2 y_3^{-1} + y_2^{-1} y_3$ has degree $6$ over $\Q$, a contradiction. We conclude that $\gcd(\ord(y_j) : j \in \{1,2,3\}) \in \{7,9\}$, and thus $n \in \{2^3 \cdot 7, 2^3 \cdot 9\}$.
\end{proof}

\begin{lem} \label{lem:2pow2}
If $2^2 \mid \ord(x_j)$ and $2^3 \nmid \ord(x_j)$ for all $j \in \{1,2,3\}$, then $n \in \{2^2 \cdot 7, 2^2 \cdot 9\}$.
\end{lem}

\begin{proof}
By the proof of Lemma \ref{lem:2pow4}, after replacing some of the $x_j$ with $x_j^{-1}$, and replacing the corresponding $c_j$ with $-c_j$, we can write $x_j = iy_j$ with $\ord(y_j)$ odd for all $j$. Then we can rewrite equations (\ref{eq:res1}) and (\ref{eq:res2}) as
\begin{align*}
c_1(y_1 + y_1^{-1}) + c_2(y_2 + y_2^{-1}) + c_3(y_3 + y_3^{-1}) &= 0 \label{eq:2pow2} \\
c_1(y_1^2 - y_1^{-2}) + c_2(y_2^2 - y_2^{-2}) + c_3(y_3^2 - y_3^{-2}) &= 0 .
\end{align*}
There is a field automorphism $\sig$ of $\Q(y_1,y_2,y_3)$ such that $\sig(y_j) = y_j^2$ for all $j$. Since $K$ is a cubic Galois extension of $\Q$, the fixed field of $\sig^3$ contains $K$. Applying $\sig^3$ to the first equation above gives us
\be
c_1(y_1^8 + y_1^{-8}) + c_2(y_2^8 + y_2^{-8}) + c_3(y_3^8 + y_3^{-8}) = 0 .
\ee
Let ${\rm Sym}(3)$ be the group of permutations of $\{1,2,3\}$, and for $p \in {\rm Sym}(3)$, let $\sgn(p)$ denote the sign of the permutation of $p$. Since $c_1,c_2,c_3$ are nonzero, we have a vanishing determinant
\begin{equation} \label{eq:det128}
0 = \det
\left(\begin{array}{ccc}
   y_1 + y_1^{-1} & y_2 + y_2^{-1} & y_3 + y_3^{-1} \\
   y_1^2 - y_1^{-2} & y_2^2 - y_2^{-2} & y_3^2 - y_3^{-2} \\
   y_1^8 + y_1^{-8} & y_2^8 + y_2^{-8} & y_3^8 + y_3^{-8}
\end{array}\right)
= \sum_{\substack{p \in {\rm Sym}(3) \\ \eps_1,\eps_2,\eps_3 \in \{\pm 1\}}} {\rm sgn}(p) \cdot \eps_2 \cdot y_{p(1)}^{\eps_1} y_{p(2)}^{2 \eps_2} y_{p(3)}^{8 \eps_3} .
\end{equation}
Since $\ord(y_j)$ is odd, by Lemma \ref{lem:bound} we have $\ord(y_j) \mid 3^2 \cdot 7 \cdot 13$ for all $j$. Let $n_0 = 3^2 \cdot 7 \cdot 13$, let $\zeta$ be a primitive $n_0$-th root of unity, and write $y_j \in \zeta^{m_j}$ with $0 \leq m_j < n_0$. By permuting indices, we may assume that
\be
\gcd(m_1,n_0) = \min(\gcd(m_j,n_0) : j \in \{1,2,3\}) .
\ee
Then by applying a field automorphism of $\Q(y_1,y_2,y_3)$, we may assume $m_1 \mid n_0$. Lastly, we may also assume $m_2 < m_3$. Note that since $\pm 1, x_1^{\pm 1}, x_2^{\pm 1}, x_3^{\pm 1}$ are distinct, $y_1^{\pm 1}, y_2^{\pm 1}, y_3^{\pm 1}$ are distinct.

We use a computer search to find all such tuples $(y_1,y_2,y_3)$ satisfying equation (\ref{eq:det128}). We find that for every such tuple, either $y_1^3 = y_2^3 = y_3^3$, or $\ord(y_j) \in \{7,9\}$ for all $j$. The first case is ruled out by the proof of Lemma \ref{lem:bound}. In the second case, we find that the only tuples for which the right-hand side in equation (\ref{eq:cjck}) is cubic satisfy either $\ord(y_j) \mid 7$ for all $j \in \{1,2,3\}$, or $\ord(y_j) \mid 9$ for all $j \in \{1,2,3\}$. Thus, $n \in \{2^2 \cdot 7, 2^2 \cdot 9\}$.
\end{proof}

\begin{lem} \label{lem:2pow1}
Suppose that $2^2 \nmid \ord(x_j)$ for all $j \in \{1,2,3\}$. Then $n \in \{7, 2 \cdot 7, 9, 2 \cdot 9\}$.
\end{lem}

\begin{proof}
After replacing some of the $x_j$ with $-x_j$, we may assume that $\ord(x_j)$ is odd for all $j$. Equation (\ref{eq:res2}) is unchanged in the process. Then there is a field automorphism $\sig$ of $\Q(x_1,x_2,x_3)$ such that $\sig(x_j) = x_j^2$ for all $j$. If $K$ is contained in the fixed field of $\sig$, then by applying $\sig$ and $\sig^{-1}$ to equation (\ref{eq:res2}), we get
\be
c_1 (x_1^{2^m} - x_1^{-2^m}) + c_2 (x_2^{2^m} - x_2^{-2^m}) + c_3 (x_3^{2^m} - x_3^{-2^m}) = 0
\ee
for $m \in \{0,1,2\}$. Since $c_1,c_2,c_3$ are nonzero, the matrix $M_{jk} = x_j^{2^{k-1}} - x_j^{2^{k-1}}$, $j,k \in \{1,2,3\}$, has vanishing determinant, which gives us
\be
0 = \det(M) = \prod_{j=1}^3 (x_j - x_j^{-1}) \prod_{k=1}^3 (x_k + x_k^{-1} - x_{k+1} - x_{k+1}^{-1}) \sum_{\ell=1}^3 (x_\ell + x_\ell^{-1}) .
\ee
Then since $\pm 1, x_1^{\pm 1}, x_2^{\pm 2}, x_3^{\pm 3}$ are distinct, it must be that
\be
x_1 + x_1^{-1} + x_2 + x_2^{-1} + x_3 + x_3^{-1} = 0 .
\ee
Since $\ord(x_j)$ is odd for all $j$, this $\Q$-relation does not contain a primitive $\Q$-relation of length $2$. If it contains a primitive $\Q$-relation of length $3$, then after permuting indices and replacing some of the $x_j$ with $x_j^{-1}$, either $x_1 + x_1^{-1} + x_2 = 0$ or $x_1 + x_2 + x_3 = 0$. In either case, we have $\ord(x_2/x_1) = \ord(x_3/x_1) = 3$, which is impossible by the proof of Lemma \ref{lem:bound}. Lastly, if this $\Q$-relation is primitive, then since $\ord(x_j)$ is odd for all $j$, by Theorem \ref{thm:DZ} with $k = 6$ and $d = 1$, we have $\ord(x_j) \mid 3 \cdot 5$ for all $j$. Then by Lemma \ref{lem:bound}, we have $\ord(x_j) \mid 3$ for all $j$, contradicting that $x_1^{\pm 1}, x_2^{\pm 1}, x_3^{\pm 1}$ are distinct. Thus, $K$ is not contained in the fixed field of $\sig$.

Since $K$ has prime degree over $\Q$, the intersection of the fixed field of $\sig$ with $K$ must be $\Q$. Replace $c_j$ with $c_j/c_1$ for all $j$, so that $c_2,c_3 \notin \Q$. By applying $\sig^{-1}$ to equation (\ref{eq:res2}), and adding or subtracting from equation (\ref{eq:res1}) according to whether or not we replaced $x_1$ with $-x_1$ at the beginning of the proof, we get
\be
(\sig^{-1}(c_2) + \eps_2 c_2)(x_2 - x_2^{-1}) + (\sig^{-1}(c_3) + \eps_3 c_3)(x_3 - x_3^{-1}) = 0
\ee
for some $\eps_2,\eps_3 \in \{\pm 1\}$. We can carry out this procedure for each pair $j \neq k$ in $\{1,2,3\}$ to get
\be
x_j - x_j^{-1} = a_{jk} (x_k - x_k^{-1})
\ee
for some nonzero $a_{jk} \in K$. By applying $\sig$, we get an equation in $x_j^{\pm 2}$ and $x_k^{\pm 2}$, and by dividing by the equation above, we get
\be
x_j + x_j^{-1} = b_{jk}(x_k + x_k^{-1})
\ee
with $b_{jk} = \sig(a_{jk})/a_{jk} \in K$. Since $\ord(x_j)$ is odd for all $j$, we have $a_{jk} \neq \pm b_{jk}$. By adding and subtracting the above $2$ equations, we can write $x_j^{\pm 1}$ in terms of $x_k^{\pm 1}$. Multiplying the two resulting expressions gives us
\be
1 = x_j x_j^{-1} = \frac{b_{jk}^2}{4}(x_k + x_k^{-1})^2 - \frac{a_{jk}^2}{4}(x_k - x_k^{-1})^2
\ee
and it follows that
\be
x_k^2 + x_k^{-2} = \frac{2(a_{jk} + b_{jk}) - 4}{a_{jk} - b_{jk}} \in K .
\ee
By applying $\sig^{-1}$, we get that $x_j + x_j^{-1} \in K$ for all $j$. Since $K$ is cubic and $\ord(x_j)$ is odd for all $j$, either $\ord(x_j) \mid 7$ for all $j$, or $\ord(x_j) \mid 9$ for all $j$. Thus, $n \in \{7, 2 \cdot 7, 9, 2 \cdot 9\}$.
\end{proof}

\begin{proof} (of Theorem \ref{thm:bound})
By Lemmas \ref{lem:bound}, \ref{lem:2pow4}, \ref{lem:2pow3}, \ref{lem:2pow2}, and \ref{lem:2pow1}, we have $n \mid 2^3 \cdot 7$ or $n \mid 2^3 \cdot 9$. We use a computer search to find all tuples $(x_1,x_2,x_3)$ such that either $\ord(x_j) \mid 2^3 \cdot 7$ for all $j$ or $\ord(x_j) \mid 2^3 \cdot 9$ for all $j$, and such that $x_1,x_2,x_3$ and the associated circumferences $c_1,c_2,c_3$ from (\ref{eq:cjck}) form an admissible solution to equations (\ref{eq:res1}) and (\ref{eq:res2}). Write $x_j = \exp(2\pi i n_j / n)$ with $0 \leq n_j < n$ and $\gcd(n_1,n_2,n_3,n) = 1$.

First, we consider conditions (2) and (3) from Theorem \ref{thm:irred}. These conditions are symmetric under permutations of $x_1,\dots,x_g$ (applying the same permutation to $c_1,\dots,c_g$ and $h_1,\dots,h_g$ as well) and under simultaneous inversion of $x_1,\dots,x_g$. Up to permutations and simultaneous inversion, the tuples satisfying these conditions are given by $(1,3,5)$ for $n = 7$, by $(1,5,11)$ for $n = 14$, and by $(1,5,14)$, $(1,6,15)$, $(1,8,11)$, $(1,8,16)$, $(2,6,15)$, $(2,7,13)$, $(2,7,14)$, $(3,6,13)$, $(3,6,14)$, $(3,7,12)$, $(3,8,12)$, $(4,8,13)$ for $n = 18$. There are no tuples satisfying these conditions for the other possible values of $n$.

Next, we additionally consider conditions (1) and (4) from Theorem \ref{thm:irred}, which are not symmetric. We find that one of the following holds.
\begin{enumerate}
    \item $n = 7$ and $(n_1,n_2,n_3)$ is one of
    \be
    (1,3,5), (1,5,3), (5,3,1) .
    \ee
    \item $n = 14$ and $(n_1,n_2,n_3)$ is one of
    \be
    (1,11,5), (11,5,1) .
    \ee
    \item $n = 18$ and $(n_1,n_2,n_3)$ is one of
    \begin{align*}
    & (1,5,14), (1,11,8), (1,14,5), (1,15,6), (1,16,8), (2,7,13), (2,13,7), (2,14,7), (2,15,6), \\ & (3,6,13), (3,7,12), (3,12,7), (3,12,8), (3,13,6), (3,14,6), (4,8,13), (4,13,8), (5,1,14), \\ & (5,14,1), (6,3,14), (6,13,3), (6,14,3), (7,2,14), (7,13,2), (7,14,2), (8,13,4), (12,7,3), \\ & (13,7,2), (14,1,5), (14,2,7), (14,5,1), (15,6,1).
    \end{align*}
\end{enumerate}
In particular, $n \in \{7,14,18\}$. If $n \in \{7,14\}$, then the unique cubic subfield of $\Q(\zeta_n)$ is $K = \Q(\cos(\pi/7))$, and if $n = 18$, then the unique cubic subfield of $\Q(\zeta_n)$ is $K = \Q(\cos(\pi/9))$.
\end{proof}


\section{Candidate Holomorphic $1$-Forms} \label{sec:main}

In this section, we apply Theorem \ref{thm:Veech} repeatedly to rule out all of the candidate stable forms from Section \ref{sec:irred} except those arising from the Veech $14$-gon, and we conclude the proof of Theorem \ref{thm:14gon}.

First, we define a notion of twist parameter of a horizontal cylinder specific to our situation. Suppose that $(X,\om) \in \Om\cM_g(g-1,g-1)^{\rm hyp}$ has an irreducible periodic horizontal direction as in Lemma \ref{lem:hypirred}. In particular, each horizontal cylinder has exactly $2$ saddle connections in each boundary, and every horizontal saddle connection has distinct endpoints. Let $C$ be a horizontal cylinder on $(X,\om)$, and let $c$ be its circumference. Choose a saddle connection $\gam \subset C \cup Z(\om)$ that is a loop and that crosses $C$ from bottom to top. The {\em twist parameter} of $C$ is $\re \int_\gam \om \in \R / c\Z$. In this setting, the twist parameter depends only on $C$ and not on the choice of $\gam$.

\begin{thm} \label{thm:cand}
Suppose $(X_0,\om_0) \in \Om\cM_3(2,2)^{\rm hyp}$ is an algebraically primitive Veech surface. There exists $(X,\om) \in \GL^+(2,\R) \cdot (X_0,\om_0)$ such that the horizontal direction is periodic and irreducible with cylinders $C_1,C_2,C_3$ with heights $h_1,h_2,h_3$, circumferences $c_1,c_2,c_3$, and twist parameters $t_1,t_2,t_3$, satisfying the following conditions.
\begin{enumerate}
    \item We have $c_1 = 1$, $h_1 = 1$, and the circumferences $1,c_2,c_3$ are part of an admissible solution from the list in the proof of Theorem \ref{thm:bound}.
    \item The twist parameters satisfy $t_1 \in \Q/\Z$, $t_2 = 0$, and $t_3 \in c_3 \Q / c_3 \Z$.
    \item With $s$ as in Theorem \ref{thm:irred}, there are integers $0 \leq p < q$ such that $0 \leq s - p/q < 1/q$ and
    \be
    \frac{(qs - p)(q + (q-p)h_2)}{((p+1) - qs)(q + (q-p-1)h_2)} \in \Q .
    \ee
    \item The numbers $s - 1, h_2, sh_2 + 1 \in K$ are $\Q$-linearly dependent.
\end{enumerate}
\end{thm}

\begin{proof}
Choose a holomorphic $1$-form in $\GL^+(2,\R) \cdot (X_0,\om_0)$ with a horizontal saddle connection with distinct endpoints. By Theorem \ref{thm:Veech}, this holomorphic $1$-form has an irreducible periodic horizontal direction. The upper-triangular subgroup of $\GL^+(2,\R)$ preserves the horizontal direction, and the upper-triangular unipotent subgroup preserves horizontal cylinder heights and circumferences. By applying a diagonal matrix, we can arrange that $h_1 = 1$ and $c_1 = 1$. Then by applying an upper-triangular unipotent matrix, we can arrange that $t_2 = 0$. Let $(X,\om)$ be the resulting holomorphic $1$-form.

The claim in (1) is now immediate from the proof of Theorem \ref{thm:bound}. For the claim in (2), since $t_2 = 0$, the cylinder $C_2$ contains a pair of vertical saddle connections $\gam_1,\gam_2$ that are homologous loops. One component of $X \sm (\gam_1 \cup \gam_2)$ contains $C_1$, and the other component contains $C_3$. Let $\al_1 \subset C_1$ and $\al_3 \subset C_3$ be closed geodesics. By Theorem \ref{thm:Veech}, the vertical foliation is periodic. By Remark \ref{rmk:distinct}, since $\gam_1$ and $\gam_2$ are loops, the vertical direction must be reducible. The first-return map of the upward vertical flow on $\al_1$ is a rotation by $-t_1$. Since this rotation is periodic, $t_1 \in \Q/\Z$. Similarly, $t_3 \in c_3 \Q / c_3 \Z$.

For the claim in (3), abusing notation, choose a representative
\be
t_1 = \frac{k}{q} \in \Q \cap [0,1)
\ee
with $k \in \Z_{\geq 0}$, $q \in \Z_{>0}$, and $\gcd(k,q) = 1$. With $s$ as in Theorem \ref{thm:irred}, since $s$ is the length of a saddle connection $\gam_0$ in the boundary of a cylinder $C_1$ with circumference $c_1 = 1$, we can write
\be
s = \frac{p}{q} + r
\ee
with $p \in \Z_{\geq 0}$ and $0 \leq r < 1/q$. Note that $0 \leq k,p < q$. Since the vertical direction is reducible, there are exactly $2$ vertical cylinders in each component of $X \sm (\gam_1 \cup \gam_2)$. Let $D_1,D_2$ be the vertical cylinders in the component containing $C_1$. One of these cylinders, $D_1$, contains both the left and right ends of the interior of the saddle connection $\gam_0$. Let $\gam_0^\pr$ be the other saddle connection in the bottom boundary of $C_1$. The cylinder $D_1$ passes through $\gam_0$ a total of $p + 1$ times, and passes through $\gam_0^\pr$ a total of $q - (p + 1)$ times. The cylinder $D_2$ passes through $\gam_0$ a total of $p$ times, and passes through $\gam_0$ a total of $q - p$ times. Since the heights $h_1^\pr,h_2^\pr$ of $D_1,D_2$ satisfy $h_1^\pr + h_2^\pr = 1/q$, we have
\be
h_1^\pr = r = s - \frac{p}{q}, \quad h_2^\pr = \frac{1}{q} - r = \frac{p+1}{q} - s .
\ee
The circumferences $c_1^\pr,c_2^\pr$ of $D_1,D_2$ are given by
\begin{align*}
c_1^\pr &= (p+1) + (q-(p+1))(1+h_2) = q + (q-p-1)h_2, \\
c_2^\pr &= p + (q-p)(1+h_2) = q + (q-p)h_2 .
\end{align*}
By Theorem \ref{thm:Veech}, the ratio
\be
\frac{h_1^\pr c_2^\pr}{h_2^\pr c_1^\pr} = \frac{(qs - p)(q + (q-p)h_2)}{((p+1) - qs)(q + (q-p-1)h_2)}
\ee
is rational.

Fix $t \in K \sm \Q$, so $(1,t,t^2)$ is a $\Q$-basis for $K$. Write
\be
s = a_0 + a_1 t + a_2 t^2, \quad h_2 = b_0 + b_1 t + b_2 t^2, \quad sh_2 = k_0 + k_1 t + k_2 t^2,
\ee
with $a_j,b_j,k_j \in \Q$, and let
\be
u = 1 + \frac{h_1^\pr c_2^\pr}{h_2^\pr c_1^\pr} = \frac{q + (q - 2p - 1)h_2 + qsh_2}{(p+1)q - q^2 s + (p+1)(q-p-1)h_2 - q(q-p-1)sh_2} .
\ee
Note that $u \in \Q$ and $u > 1$. By considering the components of $1,t,t^2$ in the numerator and denominator above, we can rewrite this rationality constraint as a matrix equation
\begin{equation} \label{eq:mat}
\begin{bmatrix}
-2b_0 & 1+b_0+k_0 & b_0 \\
-2b_1 & b_1+k_1 & b_1 \\
-2b_2 & b_2+k_2 & b_2
\end{bmatrix}
\begin{bmatrix}
p+1 \\
q \\
1
\end{bmatrix}
= u
\begin{bmatrix}
-b_0 & 1+b_0+k_0 & -a_0-k_0 \\
-b_1 & b_1+k_1 & -a_1-k_1 \\
-b_2 & b_2+k_2 & -a_2-k_2
\end{bmatrix}
\begin{bmatrix}
(p+1)^2 \\
(p+1)q \\
q^2
\end{bmatrix} .
\end{equation}
Let $M_L$ be the $3$-by-$3$ matrix on the left-hand side, and let $M_R$ be the $3$-by-$3$ matrix on the right-hand side. Clearly, $M_L$ is not invertible. Suppose that $M_R$ is invertible. Then since
\be
M_L
\begin{bmatrix}
1 \\
0 \\
0
\end{bmatrix}
=
M_R
\begin{bmatrix}
2 \\
0 \\
0
\end{bmatrix} , \quad
M_L
\begin{bmatrix}
0 \\
1 \\
0
\end{bmatrix}
=
M_R
\begin{bmatrix}
0 \\
1 \\
0
\end{bmatrix} , \quad
M_L
\begin{bmatrix}
1 \\
0 \\
2
\end{bmatrix}
=
\begin{bmatrix}
0 \\
0 \\
0
\end{bmatrix} ,
\ee
it must be that
\be
M_R^{-1} M_L =
\begin{bmatrix}
2 & 0 & -1 \\
0 & 1 & 0 \\
0 & 0 & 0
\end{bmatrix} .
\ee
Then we have
\be
u
\begin{bmatrix}
(p+1)^2 \\
(p+1)q \\
q^2
\end{bmatrix}
=
M_R^{-1} M_L
\begin{bmatrix}
p+1 \\
q \\
1
\end{bmatrix}
=
\begin{bmatrix}
2 & 0 & -1 \\
0 & 1 & 0 \\
0 & 0 & 0
\end{bmatrix}
\begin{bmatrix}
p+1 \\
q \\
1
\end{bmatrix}
=
\begin{bmatrix}
2p+1 \\
q \\
0
\end{bmatrix}
\ee
and in particular, $uq^2 = 0$. Since $q > 0$ and $u > 1$, this is a contradiction. Therefore, $M_R$ is not invertible, meaning $h_2$, $1 + h_2 + sh_2$, $-s-sh_2$ are $\Q$-linearly dependent. Equivalently,
\be
s - 1, h_2, sh_2 + 1
\ee
are $\Q$-linearly dependent.
\end{proof}

\begin{proof} (of Theorem \ref{thm:14gon})
We use the notation from Theorem \ref{thm:irred} and the proofs of Theorems \ref{thm:bound} and \ref{thm:cand}. Choose $n$ and $(n_1,n_2,n_3)$ from the list in the proof of Theorem \ref{thm:bound}. Condition (4) in Theorem \ref{thm:cand} is readily checked for each element of the list in the proof of Theorem \ref{thm:bound}. None of the elements of this list with $n = 18$ satisfy condition (4). For $n = 7$, the only tuples satisfying condition (4) are $(1,5,3)$, $(5,3,1)$, and for $n = 14$, the only tuple satisfying condition (4) is $(1,11,5)$. We now analyze the remaining $3$ admissible solutions. \\

\paragraph{\bf Case 1:} $(1,5,3)$. Letting $t = 2\cos(2\pi/7)$, we have
\be
s = \frac{9}{7} - \frac{2}{7}t - \frac{3}{7}t^2, \quad h_2 = -1 + t + t^2, \quad sh_2 = -\frac{11}{7} + \frac{4}{7}t + \frac{6}{7}t^2 ,
\ee
and $-2(s-1) = sh_2 + 1$. Subtracting row 2 from row 3 in the matrix equation (\ref{eq:mat}) and solving for $1/u$, we get $1/u = (p + 1) - q/2$. Dividing both sides of row 2 by $u$, we get
\be
\left(-2(p + 1) + \frac{11}{7}q + 1\right)\left((p + 1) - \frac{1}{2}q\right) = -(p+1)^2 + \frac{11}{7}(p+1)q - \frac{2}{7}q^2
\ee
which simplifies to
\be
(p+1)^2 - (q+1)(p+1) + \frac{1}{2}q(q+1) = 0 .
\ee
The discriminant of this quadratic equation in $p+1$ is $1 - q^2$. Since $p+1$ is rational, $1 - q^2$ must be the square of a rational number. Since $q$ is a positive integer, $q = 1$, and since $k,p$ are integers satisfying $0 \leq k,p < q$, we have $k = p = 0$. Thus, $t_1 = 0$.

The cylinder digraph $\Gam(X,\om)$ has a symmetry (reversing the chain) that swaps $C_1,C_3$ and fixes $C_2$. By this symmetry, we can carry out a similar analysis on the other component of $X \sm (\gam_1 \cup \gam_2)$. Letting $s^\pr = c_3 - c_2 + c_1 - s$, we have
\be
\frac{s^\pr}{c_3} = \frac{18}{7} - \frac{4}{7}t - \frac{6}{7}t^2, \quad \frac{h_2}{c_3} = -2 + t + t^2, \quad \frac{s^\pr h_2}{c_3^2} = -\frac{40}{7} + \frac{12}{7}t + \frac{18}{7}t^2,
\ee
and $-3(s^\pr/c_3 - 1) = s^\pr h_2/c_3^2 + 1$. Write $t_3/c_3 = k^\pr/q^\pr \in \Q \cap [0,1)$ with $k^\pr \in \Z_{\geq 0}$, $q^\pr \in \Z_{>0}$, and $\gcd(k^\pr,q^\pr) = 1$. Write $s^\pr/c_3 = p^\pr/q^\pr + r^\pr$ with $p^\pr \in \Z_{\geq 0}$ and $0 \leq r^\pr < 1/q^\pr$. Then $s^\pr/c_3$, $h_2/c_3$, $s^\pr h_2/c_3^2$ satisfy an analogous matrix equation to (\ref{eq:mat}) with $p^\pr,q^\pr$ in place of $p,q$. This time, we get
\be
(p^\pr + 1)^2 - \left(\frac{4}{3}q^\pr + 1\right)(p^\pr+1) + \frac{2}{3}q^\pr(q^\pr+1) = 0 .
\ee
The discriminant of this quadratic equation in $p^\pr+1$ is $1 - 8(q^\pr)^2/9$. Since $p^\pr + 1$ is rational, $1 - 8(q^\pr)^2/9$ must be the square of a rational number. Since $q^\pr$ is a positive integer, $q^\pr = 1$, and since $k^\pr,p^\pr$ are integers satisfying $0 \leq k^\pr,p^\pr < q^\pr$, we have $k^\pr = p^\pr = 0$. Thus, $t_3 = 0$. In this case, $(X,\om)$ lies in the $\GL^+(2,\R)$-orbit of the Veech $14$-gon. \\

\paragraph{\bf Case 2:} $(5,3,1)$. Letting $t = 2\cos(2\pi/7)$, we have
\be
s = -\frac{8}{7} + \frac{1}{7}t + \frac{5}{7}t^2, \quad h_2 = t, \quad sh_2 = \frac{5}{7} + \frac{2}{7}t - \frac{4}{7}t^2 ,
\ee
and $4(s-1) - 2h_2 + 5(sh_2 + 1) = 0$. Solving for $1/u$ in row 1 of the matrix equation (\ref{eq:mat}), we get $1/u = (p+1) + q/4$. Then since $u > 1$, we have $0 < (p+1) + q/4 < 1$, which is impossible since $p \geq 0$ and $q > 0$. \\

\paragraph{\bf Case 3:} $(1,11,5)$. Letting $t = 2\cos(\pi/7)$, we have
\be
s = \frac{8}{7} - \frac{6}{7}t + \frac{2}{7}t^2, \quad h_2 = -1 + t, \quad sh_2 = -\frac{10}{7} + \frac{18}{7}t - \frac{6}{7}t^2,
\ee
and $-3(s-1) = sh_2 + 1$. Solving for $1/u$ in row 3 of the matrix equation (\ref{eq:mat}), we get $1/u = (p+1) - 2q/3$. Dividing both sides of row 1 by $u$, after simplifying we get
\be
(p+1)^2 - \left(\frac{4}{3}q + 1\right)(p+1) + \frac{2}{3}q(q+1) = 0 ,
\ee
which as in Case 1 implies $q = 1$ and $k = p = 0$. Thus, $t_1 = 0$.

Using the symmetry of the cylinder digraph $\Gam(X,\om)$ again, writing $t_3/c_3 = k^\pr/q^\pr$ and $s^\pr/c_3 = p^\pr/q^\pr + r^\pr$ as in Case 1, we similarly get
\be
(p^\pr + 1)^2 - (q^\pr + 1)(p^\pr + 1) + \frac{1}{2}q^\pr (q^\pr + 1) = 0,
\ee
which as in Case 1 implies $q^\pr = 1$ and $k^\pr = p^\pr = 0$. Thus, $t_3 = 0$. In this case, $(X,\om)$ lies in the $\GL^+(2,\R)$-orbit of the Veech $14$-gon.
\end{proof}


\bibliographystyle{math}
\bibliography{my.bib}

{\small
\noindent
Email: kwinsor@math.harvard.edu

\noindent
Fields Institute for Research in Mathematical Sciences, Toronto, Canada
}

\end{document}